\documentclass[preprint]{elsarticle}
\usepackage{amsfonts,amsmath,amssymb,amsthm}
\usepackage{pstricks-add}
\newcommand{\Z}{\mathbb{Z}}

\newcommand{\pres}[2]{\langle {#1}\ |\ {#2} \rangle}
\newtheorem{theorem}{Theorem}
\newtheorem{lemma}[theorem]{Lemma}
\newtheorem{corollary}[theorem]{Corollary}
\newtheorem{example}[theorem]{Example}

\begin{document}
\begin{frontmatter}

\title{Tadpole Labelled Oriented Graph Groups and Cyclically Presented Groups}
\author[Howie]{James Howie}
\author[Williams]{Gerald Williams}

\address[Howie]{Department of Mathematics and the Maxwell Institute for Mathematical Sciences, Heriot-Watt University, Edinburgh EH14 4AS, U.K.}
\address[Williams]{Department of Mathematical Sciences, University of Essex, Colchester, Essex CO4 3SQ, U.K.}
\fntext[Howie]{{\tt J.Howie@hw.ac.uk}}
\fntext[Williams]{{\tt gwill@essex.ac.uk}}

\begin{abstract}
We study a class of Labelled Oriented Graph (LOG) group where the underlying graph is a tadpole graph. We show that such a group is the natural HNN extension of a cyclically presented group and investigate the relationship between the LOG group and the cyclically presented group. We relate the second homotopy groups of their presentations and show that hyperbolicity of the cyclically presented group implies solvability of the conjugacy problem for the LOG group. In the case where the label on the tail of the LOG spells a positive word in the vertices in the circuit we show that the LOGs and groups coincide with those considered by Szczepa\'{n}ski and Vesnin. We obtain new presentations for these cyclically presented groups and show that the groups of Fibonacci type introduced by Johnson and Mawdesley are of this form. These groups generalize the Fibonacci groups and the Sieradski groups and have been studied by various authors. We continue these investigations, using small cancellation and curvature methods to obtain results on hyperbolicity, automaticity, SQ-universality, and solvability of decision problems.
\end{abstract}

\end{frontmatter}

\section{Introduction}\label{sec:intro}

Continuing research carried out in~\cite{GilbertHowie},\cite{SV} we study a class of Labelled Oriented Graph (LOG) groups where the underlying graph is a tadpole graph. It will turn out that these groups are isomorphic to the natural HNN extension of cyclically presented groups. In Sections~\ref{sec:LOG}--\ref{sec:conjugacy} of this paper we explore the relationship between the LOG group and presentation and the cyclically presented group and its cyclic presentation. One example of a family of groups arising in this way are the \em groups of Fibonacci type \em ${H}_n(m,k)$ defined by the presentations
\[\mathcal{H}_n(m,k)=\pres{x_1,\ldots , x_n}{x_ix_{i+m}=x_{i+k}\ (1\leq i\leq n)}\]
($n\geq 2$, $1\leq m,k\leq n$, subscripts mod~$n$), and in Sections~\ref{sec:smallcanc}--\ref{sec:hyperbolicity} we study these groups and presentations. This continues research into the groups conducted since their introduction in~\cite{JohnsonMawdesley} and their (independent) re-introduction in~\cite{CHR}, where they were considered as generalizations of other classes of groups considered in the literature. They include the presentations $\mathcal{F}(2,n)=\mathcal{H}_n(1,2)$ of Conway's Fibonacci groups $F(2,n)$, the presentations $\mathcal{S}(2,n)=\mathcal{H}_n(2,1)$ of the \em Sieradski  groups \em $S(2,n)$, and more generally the presentations $\mathcal{H}(n,m)=\mathcal{H}_n(m,1)$ of the groups $H(n,m)$ considered in~\cite{GilbertHowie}. See~\cite{Wsurvey} for a survey of research on the algebraic properties of the groups $H_n(m,k)$.

\section{Preliminaries}\label{sec:prelims}

A \em Labelled Oriented Graph (LOG) \em  consists of a finite connected graph (possibly with loops and multiple edges) with vertex set $V$ and edge set $E$ together with three maps $\iota,\tau,\lambda:E \rightarrow V$ called the \em initial vertex map, terminal vertex map\em , and \em labelling map\em , respectively. The LOG determines a corresponding \em LOG presentation\em
\[\pres{V}{\tau(e)^{-1}\lambda(e)^{-1}\iota(e)\lambda(e)\ (e\in E)}.\]
A group with a LOG presentation is a \em LOG group. \em

In~\cite{GilbertHowie} a particular LOG was considered: this consisted of a directed circuit of $n$ vertices with one additional edge, incident with one vertex of the circuit and with a unique extremal vertex, such that all edges of the circuit are labelled by this extremal vertex. (That investigation was prompted by a remark of Bogley~\cite{Bogley} about the case $n=2$ and led to the first detailed considerations of the groups $H(n,m)$.) A natural way to generalize that LOG is to replace the additional edge by a directed path where the labels of the edges appearing in that path spell a word in the labels of the vertices that appear in the circuit (see Figure~\ref{fig:TadpoleLOG}) and that is the situation we consider in this paper. It turns out that when the word on the tail is a positive word the LOGs coincide with a class of LOGs considered in~\cite{SV}.

Let $w$ be an element of the free group $F_n$ on $\{x_1,\ldots ,x_n\}$ and $\theta$ be the automorphism of $F_n$ given by $\theta(x_i)=x_{i+1}$, $\theta(x_n)=x_1$. Then
\[\mathcal{G}_n(w) =\pres{x_1,\ldots ,x_n}{w,\theta(w),\ldots ,\theta^{n-1}(w)}\]
is called a \em cyclic presentation \em and the group it defines, $G_n(w)$, is a \em cyclically presented group\em. The automorphism $\theta$ of $F_n$ induces an automorphism $\phi$ of $G_n(w)$ given by $\phi(x_i)=x_{i+1}$, $\phi(x_n)=x_1$.

Following~\cite{SV} we call the semidirect product
\[\widehat{G}_n(w)=G_n(w)\rtimes_\phi \Z=G_n(w)*_\phi=\{G_n(w),t\ |\ t^{-1}gt=\phi(g), g\in G_n(w)\}\]
the \em natural HNN extension \em of $G_n(w)$.

\begin{lemma}\label{lem:HNNpres}
The natural HNN extension $\widehat{G}_n(w)$ of a cyclically presented group ${G}_n(w)$ has a presentation $\widehat{\mathcal{G}}_n(w)=\pres{a,c}{W(a,c),[a,c^n]}$
where
\[W(a,c)= w(ac^{-1},cac^{-2},c^2ac^{-3},\ldots , c^{n-1}ac^{-n}).\]
\end{lemma}

\begin{proof}
\begin{alignat*}{1}
\widehat{G}_n(w)&= \langle x_1,\ldots ,x_n,t\ |\ x_1=t^{-1}x_nt, x_{i+1}=t^{-1}x_it\ (1\leq i \leq n-1), \\
&\qquad \qquad \qquad \qquad \qquad \qquad w(x_1,\ldots ,x_n) \rangle\\
&=\langle x_1,\ldots ,x_n,t\ |\ x_1=t^{-n}x_1t^n, x_{i+1}=t^{-i}x_1t^i, \\
&{}\qquad \qquad \qquad \qquad \qquad \qquad w(x_1,t^{-1}x_1t,t^{-2}x_1t^2, \ldots ,t^{-(n-1)}x_1t^{n-1})\rangle \\
&=\pres{x_1,t}{x_1=t^{-n}x_1t^n, w(x_1,t^{-1}x_1t,t^{-2}x_1t^2, \ldots ,t^{-(n-1)}x_1t^{n-1})}\\
&=\langle x_1,t,a,c\ |\ x_1=t^{-n}x_1t^n, w(x_1,t^{-1}x_1t,t^{-2}x_1t^2, \ldots ,t^{-(n-1)}x_1t^{n-1}),\\
&{}\qquad \qquad \qquad \qquad \qquad \qquad c=t^{-1},a=x_1c \rangle\\
&=\pres{a,c}{ac^{-1}=c^{n}ac^{-n-1}, w(ac^{-1},cac^{-2},c^2ac^{-3}, \ldots ,c^{n-1}ac^{-n})}\\
&=\pres{a,c}{W(a,c),[a,c^n]}
\end{alignat*}
\end{proof}

The defining word $w$ of $\mathcal{G}_n(w)$ is \em admissible \em if the exponent sum of its generators is $\pm 1$. Szczepa\'{n}ski and Vesnin~\cite[Theorem~1]{SV} showed that the natural HNN extension $\widehat{G}_n(w)$ satisfies the Kervaire conditions for it to be an $l$-knot group ($l\geq 3$) if and only if $w$ is admissible. In this case, by Simon~\cite{Simon}, it has a Wirtinger presentation and it is the fundamental group of the complement $S^{l+2}\backslash M^l$ where $M^l$ is a closed orientable connected tamely embedded $l$-manifold ($l\geq 2$).

Associated to a group presentation is the standard 2-complex which has a single 0-cell, 1-cells in one-to-one correspondence with the generators, and 2-cells in one-to-one correspondence with the relators. We will not distinguish between the presentation and the 2-complex, and we will talk of homotopy equivalences and homotopy groups of presentations. We will use the symbol $\simeq$ to denote homotopy equivalence.

\section{LOG groups and natural HNN extensions}\label{sec:LOG}

We consider LOGs defined as follows. Let $v$ be a freely reduced word of length $r$ in generators $a_1,\ldots ,a_n$. Then we may write
\[v(a_1,\ldots ,a_n)= a_{n-p_0}^{\delta_0} a_{n-p_1}^{\delta_1}\ldots a_{n-p_{r-1}}^{\delta_{r-1}}\]
for some $0\leq p_0,\ldots , p_{r-1} \leq n-1$, $\delta_0,\ldots ,\delta_{r-1}\in\{1,-1\}$, where $\delta_i\neq-\delta_{i+1}$ if $p_i=p_{i+1}$. (It is convenient to write the subscripts in the form $n-p_i$, rather than simply $p_i$, for reasons that will become clear in Corollary~\ref{cor:poswordcycpres}.)
Let $\Gamma(n; v(a_1,\ldots ,a_n))$ be the LOG depicted in Figure~\ref{fig:TadpoleLOG} (note that $\Gamma(n; a_{n-m})$ is the LOG considered in~\cite{GilbertHowie}).

\begin{figure}[ht]
\vspace{-3.0cm}
  \begin{center}
    \begin{minipage}[t]{0.58\linewidth}
\psset{xunit=1.1cm,yunit=1.2cm,algebraic=true,dotstyle=o,dotsize=3pt 0,linewidth=0.8pt,arrowsize=3pt 2,arrowinset=0.25}
\begin{pspicture*}(-2.2,-6.8)(19.6,5.8)
\psline(-2,1)(-1,1)
\psline(-1.4,1)(-1.5,0.9)
\psline(-1.4,1)(-1.5,1.1)
\psline(3,1)(4,1)
\psline(3.6,1)(3.5,0.9)
\psline(3.6,1)(3.5,1.1)
\psline[linestyle=dashed,dash=2pt 2pt](5,2.7)(6,3)
\psline[linestyle=dashed,dash=2pt 2pt](4.9,-0.8)(6,-1)
\psline[linestyle=dashed,dash=2pt 2pt](8,1.5)(7.7,2.4)
\psline[linestyle=dashed,dash=2pt 2pt](8,0.4)(7.6,-0.4)
\psline[linestyle=dashed,dash=2pt 2pt](-0.6,1)(0.1,1)
\psline(0.5,1)(1.5,1)
\psline(1.1,1)(1,0.9)
\psline(1.1,1)(1,1.1)
\psline[linestyle=dashed,dash=2pt 2pt](1.9,1)(2.7,1)
\psline(4.3,2)(4,1)
\psline(4.1,1.4)(4,1.5)
\psline(4.1,1.4)(4.3,1.5)
\psline(5,2.7)(4.3,2)
%\psline(4.5,2.3)(4.5,2.4)
%\psline(4.5,2.3)(4.7,2.2)
\psline(4.55,2.25)(4.55,2.45)%
\psline(4.55,2.25)(4.75,2.25)%
\psline(4.3,0)(4.9,-0.8)
\psline(4.7,-0.5)(4.5,-0.5)
\psline(4.7,-0.5)(4.7,-0.3)
\psline(4,1)(4.3,0)
\psline(4.2,0.4)(4,0.5)
\psline(4.2,0.4)(4.3,0.6)
\psline(8,0.4)(8,1.5)
\psline(8,1.1)(8.1,1)
\psline(8,1.1)(7.9,1)
\psdots[dotstyle=*,linecolor=blue](-2,1)
\rput[bl](-2.1,1.2){\blue{$t_0$}}
\psdots[dotstyle=*,linecolor=blue](-1,1)
\rput[bl](-1.1,1.2){\blue{$t_1$}}
\rput[bl](-1.7,0.3){$a_{n-p_{0}}^{\delta_{0}}$}
\psdots[dotstyle=*,linecolor=blue](3,1)
\rput[bl](2.9,1.2){\blue{$t_{r-1}$}}
\psdots[dotstyle=*,linecolor=blue](4,1)
\rput[bl](4.1,0.9){\blue{$a_1$}}
\rput[bl](2.9,0.3){$a_{n-p_{r-1}}^{\delta_{r-1}}$}
\psdots[dotstyle=*,linecolor=blue](4.3,2)
\rput[bl](4.5,1.8){\blue{$a_n$}}
\psdots[dotstyle=*,linecolor=blue](5,2.7)
\rput[bl](5.2,2.4){\blue{$a_{n-1}$}}
\psdots[dotstyle=*,linecolor=blue](4.3,0)
\rput[bl](4.4,0){\blue{$a_2$}}
\psdots[dotstyle=*,linecolor=blue](4.9,-0.8)
\rput[bl](5,-0.7){\blue{$a_3$}}
\psdots[dotstyle=*,linecolor=blue](8,0.4)
\rput[bl](7.5,0.3){\blue{$a_i$}}
\psdots[dotstyle=*,linecolor=blue](8,1.5)
\rput[bl](7.3,1.4){\blue{$a_{i+1}$}}
\psdots[dotstyle=*,linecolor=blue](0.5,1)
\rput[bl](0.4,1.2){\blue{$t_j$}}
\psdots[dotstyle=*,linecolor=blue](1.5,1)
\rput[bl](1.4,1.2){\blue{$t_{j+1}$}}
\rput[bl](0.7,0.3){$a_{n-p_{j}}^{\delta_{j}}$}
\rput[bl](3.8,1.6){$t_0$}
\rput[bl](4.2,2.5){$t_0$}
\rput[bl](4.2,-0.7){$t_0$}
\rput[bl](3.95,0.1){$t_0$}
\rput[bl](8.2,0.9){$t_0$}
\end{pspicture*}
    \end{minipage}\hfill
\vspace{-6.5cm}
\caption{LOG $\Gamma(n; v(a_1,\ldots ,a_n))$.\label{fig:TadpoleLOG}}
  \end{center}
\normalsize
\end{figure}
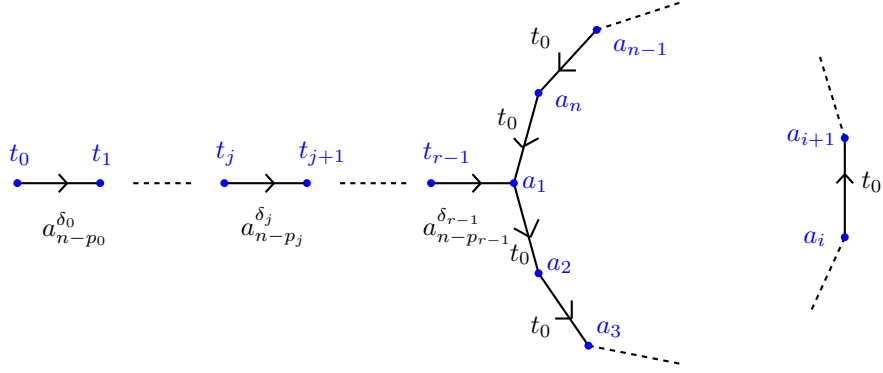

\begin{lemma}\label{lem:LOGpres}
The LOG presentation corresponding to $\Gamma(n; v(a_1,\ldots ,a_n))$ is homotopically equivalent to
\( \pres{a,c}{U(a,c), [a,c^n]}\)
where
\begin{alignat*}{1}
U(a,c) &=v(a,c^{-1}ac,c^{-2}ac^2,\ldots ,c^{-(n-1)}ac^{(n-1)})^{-1}c\\
&{}\qquad\qquad \cdot \ v(a,c^{-1}ac,c^{-2}ac^2,\ldots ,c^{-(n-1)}ac^{(n-1)})a^{-1}
\end{alignat*}
and where $c=t_0$ and $a=a_1$.
\end{lemma}

\begin{proof}
The LOG presentation is
\begin{alignat*}{1}
&{}       \langle a_1,\ldots , a_n,t_0,\ldots, t_{r-1}\ |\
t_j=a_{n-p_j}^{\delta_j}t_{j+1}a_{n-p_j}^{-\delta_j}\ (j=0,\ldots ,r-2),\\
&{}\qquad \qquad \qquad \qquad t_{r-1}=a_{n-p_{r-1}}^{\delta_{r-1}}a_1a_{n-p_{r-1}}^{-\delta_{r-1}}, a_i=t_0a_{i+1}t_0^{-1}\ (i=1,\ldots ,n)
\rangle\\
&\simeq \langle a_1,\ldots , a_n,t_0,\ldots, t_{r-1}\ |\
t_{r-1}=a_{n-p_{r-1}}^{\delta_{r-1}}a_1a_{n-p_{r-1}}^{-\delta_{r-1}},\\
&{}\qquad \qquad \qquad \qquad t_{j+1}=\left(\prod_{\alpha=0}^j a_{n-p_\alpha}^{\delta_\alpha}\right)^{-1}t_0 \left(\prod_{\alpha=0}^j a_{n-p_\alpha}^{\delta_\alpha}\right) (j=0,\ldots ,r-2),\\
&{}\qquad \qquad \qquad \qquad a_i=t_0a_{i+1}t_0^{-1}\ (i=1,\ldots ,n)\rangle \\
&\simeq \langle a_1,\ldots , a_n,t_0\ |\
\left(\prod_{\alpha=0}^{r-2} a_{n-p_\alpha}^{\delta_\alpha}\right)^{-1} t_0 \left(\prod_{\alpha=0}^{r-2} a_{n-p_\alpha}^{\delta_\alpha}\right) = a_{n-p_{r-1}}^{\delta_{r-1}} a_1 a_{n-p_{r-1}}^{-\delta_{r-1}},\\
&{}\qquad \qquad \qquad \qquad a_i=t_0a_{i+1}t_0^{-1}\ (i=1,\ldots ,n)
\rangle\\
&\simeq \langle a_1,\ldots , a_n,t_0\ |\
\left(\prod_{\alpha=0}^{r-1} a_{n-p_\alpha}^{\delta_\alpha}\right)^{-1} t_0 \left(\prod_{\alpha=0}^{r-1} a_{n-p_\alpha}^{\delta_\alpha}\right) = a_1,\\
&{}\qquad \qquad \qquad \qquad a_i=t_0a_{i+1}t_0^{-1}\ (i=1,\ldots ,n)
\rangle\\
&\simeq \langle a_1,\ldots , a_n,t_0\ |\
v(a_1,\ldots ,a_n)^{-1} t_0 v(a_1,\ldots ,a_n) = a_1,\\
&{}\qquad \qquad \qquad \qquad a_i=t_0a_{i+1}t_0^{-1}\ (i=1,\ldots ,n)
\rangle\displaybreak[2]\\
&\simeq \langle a_1,\ldots , a_n,t_0\ |\
v(a_1,\ldots ,a_n)^{-1} t_0 v(a_1,\ldots ,a_n) = a_1,\\
&{}\qquad \qquad \qquad \qquad a_i=t_0^{-(i-1)}a_1t_0^{i-1}\ (i=2,\ldots , n), a_1=t_0^{-n}a_1t_0^n
\rangle\\
&\simeq \pres{a,c}{U(a,c), [a,c^n]}
\end{alignat*}
\end{proof}

In the corollary we consider the case when $v$ is a positive word; that is $\delta_0=\cdots =\delta_{r-1}=1$ so  $v(a_1,\ldots ,a_n)= a_{n-p_0}a_{n-p_1}\ldots a_{n-p_{r-1}}$.
\begin{corollary}\label{cor:poswordcycpres}
The LOG group corresponding to $\Gamma(n; a_{n-p_0}a_{n-p_1}\ldots a_{n-p_{r-1}})$ has a presentation
\begin{alignat*}{1}
\pres{a,c}{\left( \prod_{i=0}^{r-1} ac^{p_{i+1}-p_i} \right)^{-1} c \left( \prod_{i=0}^{r-1} ac^{p_{i+1}-p_i} \right)a^{-1}, [a,c^n]}
\end{alignat*}
and is isomorphic to the natural HNN extension of
\begin{alignat*}{1}
 G_n((x_{p_0}x_{p_1+1}\ldots x_{p_{r-1}+(r-1)} x_{p_{r}+r}) (x_{p_0+1}x_{p_1+2}\ldots x_{p_{r-1}+r} )^{-1} )\label{eq:poswordCPgp}
\end{alignat*}
where $p_r=-1$.
\end{corollary}
\begin{proof}
In the notation of Lemma~\ref{lem:LOGpres}
\begin{alignat*}{1}
 v(a_1,\ldots , a_n) &= \prod_{i=0}^{r-1} a_{n-p_i}
\intertext{so}
 v(a,c^{-1}ac, \ldots ,c^{-(n-1)}ac^{n-1}) &= \prod_{i=0}^{r-1} c^{-(n-p_i-1)}ac^{n-p_i-1}\\
 &=\prod_{i=0}^{r-1} c^{(p_i+1)}ac^{-(p_i+1)}\\
 &=c^{p_0+1}\prod_{i=0}^{r-1}ac^{p_{i+1}-p_i}
\end{alignat*}
(using the fact that $c^{-n}ac^n=a$) and the presentation of the LOG group follows.

Let
\[w(x_1,\ldots ,x_n)= (x_{p_0}x_{p_1+1}\ldots x_{p_{r-1}+(r-1)} x_{p_{r}+r}) (x_{p_0+1}x_{p_1+2}\ldots x_{p_{r-1}+r} )^{-1}.\]
Then this is equal to
$w'(x_1,x_2,\ldots ,x_{n-1},x_n) x_{p_{r}+r} w'(x_2,x_3,\ldots ,x_{n},x_1)^{-1}$ where
\[w'(x_1,x_2,\ldots ,x_{n-1},x_n) = x_{p_0}x_{p_1+1}\ldots x_{p_{r-1}+(r-1)}\]
and so
\begin{alignat*}{1}
 &{} w(ac^{-1},cac^{-2},c^2ac^{-3},\ldots , c^{n-1}ac^{-n}) \\
 &= w'(ac^{-1},cac^{-2},c^2ac^{-3},\ldots , c^{n-1}ac^{-n}) c^{p_r+r-1}ac^{-(p_r+r)}\\
 &{}\qquad \qquad \qquad  \cdot \ w'(cac^{-2},c^2ac^{-3},c^3ac^{-4},\ldots , c^{n}ac^{-(n+1)})^{-1}\\
 &= w'(ac^{-1},cac^{-2},c^2ac^{-3},\ldots , c^{n-1}ac^{-n}) c^{p_r+r-1}ac^{-(p_r+r-1)} \\
 &{}\qquad \qquad \qquad  \cdot \ w'(ac^{-1},cac^{-2},c^2ac^{-3},\ldots , c^{n-1}ac^{-n})^{-1}c^{-1}.
\end{alignat*}
Now
\begin{alignat*}{1}
&{} w'(ac^{-1},cac^{-2},c^2ac^{-3},\ldots , c^{n-1}ac^{-n})\\
&=
c^{p_0-1}ac^{p_1-p_0}ac^{p_2-p_1}\ldots ac^{p_{r-1}-p_{r-2}}ac^{-p_{r-1}-(r-1)}\\
&= c^{p_0-1}\left(\prod_{i=1}^r ac^{p_i-p_{i-1}} \right) c^{-(r-2)}
\end{alignat*}
so
\begin{alignat*}{1}
 &{} w(ac^{-1},cac^{-2},c^2ac^{-3},\ldots , c^{n-1}ac^{-n}) \\
 &= c^{p_0-1}\left(\prod_{i=1}^r ac^{p_i-p_{i-1}} \right)
 c^{-r+2} c^{p_r+r-1}ac^{-p_r-r+1}
 c^{r-2}\left(\prod_{i=1}^r ac^{p_i-p_{i-1}} \right) ^{-1}c^{1-p_0}\\
 &= c^{p_0-1}\left(\prod_{i=1}^r ac^{p_i-p_{i-1}} \right)
 c^{p_r+1}ac^{-(p_r+1)}
 \left(\prod_{i=1}^r ac^{p_i-p_{i-1}} \right) ^{-1}c^{1-p_0}\\
 &= c^{p_0-1}\left(\prod_{i=1}^r ac^{p_i-p_{i-1}} \right)
 a
 \left(\prod_{i=1}^r ac^{p_i-p_{i-1}} \right) ^{-1}c^{1-p_0}
\end{alignat*}
(since $p_r=-1$), and this is a conjugate of the first relator in the LOG presentation, so an application of Lemma~\ref{lem:HNNpres} completes the proof.
\end{proof}

We shall give the following notation to the presentation in Corollary~\ref{cor:poswordcycpres}
\begin{alignat*}{1}
&{}\mathcal{L}_n(p_0,p_1,\ldots ,p_{r-1})\\
&=\mathcal{G}_n((x_{p_0}x_{p_1+1}\ldots x_{p_{r-1}+(r-1)} x_{p_{r}+r}) (x_{p_0+1}x_{p_1+2}\ldots x_{p_{r-1}+r} )^{-1} )
\end{alignat*}
and write $L_n(p_0,p_1,\ldots ,p_{r-1})$ for the group it defines. In~\cite{SV} the following family of cyclic presentations
\begin{alignat*}{1}
\mathcal{G}_{n,k}^r(\overline{q};\overline{\epsilon})&=\mathcal{G}_{n,k}^r(q_1,\ldots ,q_r;\epsilon_0,\ldots ,\epsilon_r)\\
 &=  \mathcal{G}_n(x_1^{\epsilon_0}x_{1+q_1}^{\epsilon_1}\ldots x_{1+q_{r-1}}^{\epsilon_{r-1}} x_{1+q_{r}}^{\epsilon_{r}}(x_{1+k}^{\epsilon_0}x_{1+k+q_1}^{\epsilon_1}\ldots x_{1+k+q_{r-1}}^{\epsilon_{r-1}} )^{-1})
\end{alignat*}
and the groups ${G}_{n,k}^r(\overline{q};\overline{\epsilon})$ they define were introduced (where $n\geq 2$, $r\geq 1$, $k,q_1,\ldots , q_r,\epsilon_0,\epsilon_1,\ldots ,\epsilon_r$ are integers). Amongst others, the groups $H_n(m,k)=G_{n,k}^1(m;1,1)$ are given there as examples of groups belonging to this family. The subfamily $G_{n,1}^r(\overline{q};\overline{1})=G_{n,1}^r(q_1,\ldots ,q_r;1,\ldots,1)$ was singled out for special attention in that paper as their natural HNN extensions are LOG groups. Actually, this subfamily coincides with our presentations $\mathcal{L}_n(p_0,p_1,\ldots ,p_{r-1})$. To see this, note
\begin{alignat*}{1}
&{} \mathcal{G}_{n,1}^r(1+(p_1-p_0),2+(p_2-p_0),\ldots ,(r-1)+(p_{r-1}-p_0),r+p_r-p_0;\overline{1})\\
&= \mathcal{G}_n((x_{1}x_{2+p_1-p_0}\ldots x_{r+p_{r-1}-p_0} x_{r+1+p_{r}-p_0})
(x_{2}x_{3+p_1-p_0}\ldots x_{r+1+p_{r-1}-p_0)} )^{-1} )\\
&= \mathcal{G}_n((x_{p_0}x_{p_1+1}\ldots x_{p_{r-1}+(r-1)} x_{p_{r}+r}) (x_{p_0+1}x_{p_1+2}\ldots x_{p_{r-1}+r} )^{-1} )\\
&=\mathcal{L}_n(p_0,\ldots ,p_{r-1})
\end{alignat*}
and
\begin{alignat*}{1}
&{} \mathcal{L}_n((r-1)-q_r,q_1-1+(r-1)-q_r,q_2-2+(r-1)-q_r,\ldots , \\
&{}\qquad \qquad \qquad \qquad q_{r-1}-(r-1)+(r-1)-q_r)\\
&= \mathcal{G}_n( (x_{(r-1)-q_r}x_{q_1+(r-1)-q_r}x_{q_2+(r-1)-q_r}\ldots x_{q_{r-1}+(r-1)-q_r} x_{p_r+r})\\
&{}\qquad \qquad \qquad \qquad \cdot\ (x_{r-q_r}x_{q_1+r-q_r}\ldots x_{q_{r-1}+r-q_r} )^{-1})\\
&= \mathcal{G}_n( (x_{1}x_{q_1+1}x_{q_2+1}\ldots x_{q_{r-1}+1} x_{1+q_r})
(x_{2}x_{q_1+2}x_{q_2+2}\ldots x_{q_{r-1}+2} )^{-1})\\
&= \mathcal{G}_{n,1}^r(q_1,q_2,\ldots ,q_r;\overline{1}).
\end{alignat*}

Similarly, when $\delta_0=\cdots =\delta_{r-1}=1$ the LOG in Figure~\ref{fig:TadpoleLOG} can be obtained from that in~\cite[Figure~1]{SV} by setting
$q_j = p_j-p_0+j$ $(j=1,\ldots ,r)$, $b_i=a_{n-p_0+(r-1)-j+1}$, $(i=1,\ldots ,n)$ and the LOG in~\cite[Figure~1]{SV} can be obtained from that in  Figure~\ref{fig:TadpoleLOG} by setting $p_j= n-j-(q_r-q_j)+r-1$ $(j=0,\ldots , r-1)$, $a_i=b_{q_r-(i-1)}$ $(i=1,\ldots ,n)$, and $\delta_0=\cdots =\delta_{r-1}=1$.
Clearly, and as observed in~\cite{SV}, $\mathcal{G}_{n,k}^r(q_1,\ldots ,q_r;\epsilon_0,\ldots ,\epsilon_r)$ has an admissible defining word if and only if $\epsilon_r=\pm 1$, in which case without loss of generality we may take $\epsilon_r=1$. For such cases we give, in the next theorem, an alternative cyclic presentation.

\begin{theorem}\label{thm:newSVpres}
\begin{alignat*}{1}
&{} \mathcal{G}_{n,k}^r(q_1,\ldots ,q_{r-1},q_r;\epsilon_0,\ldots ,\epsilon_{r-1},1)\\
&\simeq
\mathcal{G}_n( (y_1^{-1}y_{1+k})^{\epsilon_0} (y_{1+q_1}^{-1}y_{1+q_1+k})^{\epsilon_1} \ldots (y_{1+q_{r-1}}^{-1}y_{1+q_{r-1}+k})^{\epsilon_{r-1}} y_{1+q_r}^{-1}    )
\end{alignat*}
and, in particular,
\begin{alignat*}{1}
&{} \mathcal{L}_n(p_0,p_1,\ldots ,p_{r-1}) \simeq \\
&{} \mathcal{G}_n( (y_1^{-1}y_2)(y_{2+(p_1-p_0)}^{-1}y_{3+(p_1-p_0)})(y_{3+(p_2-p_0)}^{-1}y_{4+(p_3-p_0)})
\ldots \\
&{}\qquad \qquad \qquad \qquad (y_{r+(p_{r-1}-p_0)}^{-1}y_{r+1+(p_{r-1}-p_0)}) y_{1+r+(p_r-p_0)}^{-1})
\end{alignat*}
where $p_r=-1$.
\end{theorem}

\begin{proof}
\begin{alignat*}{1}
&{}\mathcal{G}_{n,k}^r(q_1,\ldots ,q_{r-1},q_r;\epsilon_0,\ldots ,\epsilon_{r-1},1)\\
&\simeq\pres{x_1\ldots , x_n}{x_i^{\epsilon_0}x_{i+q_1}^{\epsilon_1}\ldots x_{i+q_{r}}^{1} = x_{i+k}^{\epsilon_0}x_{i+k+q_1}^{\epsilon_1}\ldots x_{i+k+q_{r-1}}^{\epsilon_{r-1}}\ (1\leq i \leq n)}\\
&\simeq\pres{x_1\ldots , x_n,y_1,\ldots ,y_n}{y_i=x_i^{\epsilon_0}x_{i+q_1}^{\epsilon_1}\ldots x_{i+q_{r-1}}^{\epsilon_{r-1}},\ y_ix_{i+q_r}=y_{i+k}\ (1\leq i \leq n)}\\
&\simeq \langle x_1\ldots , x_n,y_1,\ldots ,y_n\ |\ y_{i+q_r}=x_{i+q_r}^{\epsilon_0}x_{i+q_1+q_r}^{\epsilon_1}\ldots x_{i+q_{r-1}+q_r}^{\epsilon_{r-1}},\\
&{}\qquad \qquad \qquad \qquad \qquad \qquad  x_{i+q_r}=y_{i}^{-1}y_{i+k}\ (1\leq i \leq n)\rangle \\
&\simeq\langle y_1,\ldots ,y_n\ |\ y_{i+q_r}=(y_{i}^{-1}y_{i+k})^{\epsilon_0}(y_{i+q_1}^{-1}y_{i+q_1+k})^{\epsilon_1}\ldots (y_{i+q_{r-1}}^{-1}y_{i+q_{r-1}+k})^{\epsilon_{r-1}}\\
&{}\qquad \qquad \qquad \qquad \qquad \qquad   (1\leq i \leq n)\rangle.
\end{alignat*}
\end{proof}

\begin{corollary}\label{cor:CHRLOG}
The natural HNN extension $\widehat{H}_n(m,k)$ of $H_n(m,k)$ is the LOG group corresponding to the LOG in Figure~\ref{fig:CHRLOG}.
\end{corollary}

\begin{proof}
\( \mathcal{L}_n( \underbrace{(k-m)-1,\ldots ,(k-m)-1}_k )) = \mathcal{H}_n(m,k). \)
\end{proof}

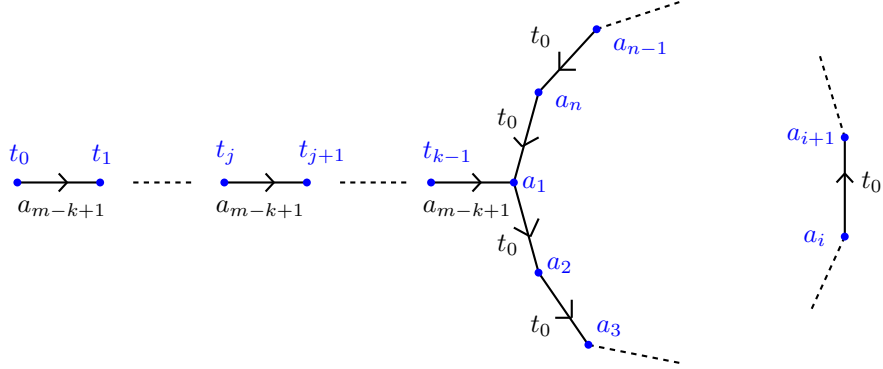
\begin{figure}[ht]
\vspace{-3.0cm}
  \begin{center}
    \begin{minipage}[t]{0.58\linewidth}
\psset{xunit=1.1cm,yunit=1.2cm,algebraic=true,dotstyle=o,dotsize=3pt 0,linewidth=0.8pt,arrowsize=3pt 2,arrowinset=0.25}
\begin{pspicture*}(-2.2,-6.8)(19.6,5.8)
\psline(-2,1)(-1,1)
\psline(-1.4,1)(-1.5,0.9)
\psline(-1.4,1)(-1.5,1.1)
\psline(3,1)(4,1)
\psline(3.6,1)(3.5,0.9)
\psline(3.6,1)(3.5,1.1)
\psline[linestyle=dashed,dash=2pt 2pt](5,2.7)(6,3)
\psline[linestyle=dashed,dash=2pt 2pt](4.9,-0.8)(6,-1)
\psline[linestyle=dashed,dash=2pt 2pt](8,1.5)(7.7,2.4)
\psline[linestyle=dashed,dash=2pt 2pt](8,0.4)(7.6,-0.4)
\psline[linestyle=dashed,dash=2pt 2pt](-0.6,1)(0.1,1)
\psline(0.5,1)(1.5,1)
\psline(1.1,1)(1,0.9)
\psline(1.1,1)(1,1.1)
\psline[linestyle=dashed,dash=2pt 2pt](1.9,1)(2.7,1)
\psline(4.3,2)(4,1)
\psline(4.1,1.4)(4,1.5)
\psline(4.1,1.4)(4.3,1.5)
\psline(5,2.7)(4.3,2)
%\psline(4.5,2.3)(4.5,2.4)
%\psline(4.5,2.3)(4.7,2.2)
\psline(4.55,2.25)(4.55,2.45)%
\psline(4.55,2.25)(4.75,2.25)%
\psline(4.3,0)(4.9,-0.8)
\psline(4.7,-0.5)(4.5,-0.5)
\psline(4.7,-0.5)(4.7,-0.3)
\psline(4,1)(4.3,0)
\psline(4.2,0.4)(4,0.5)
\psline(4.2,0.4)(4.3,0.6)
\psline(8,0.4)(8,1.5)
\psline(8,1.1)(8.1,1)
\psline(8,1.1)(7.9,1)
\psdots[dotstyle=*,linecolor=blue](-2,1)
\rput[bl](-2.1,1.2){\blue{$t_0$}}
\psdots[dotstyle=*,linecolor=blue](-1,1)
\rput[bl](-1.1,1.2){\blue{$t_1$}}
\rput[bl](-2.0,0.6){$a_{m-k+1}$}
\psdots[dotstyle=*,linecolor=blue](3,1)
\rput[bl](2.9,1.2){\blue{$t_{k-1}$}}
\psdots[dotstyle=*,linecolor=blue](4,1)
\rput[bl](4.1,0.9){\blue{$a_1$}}
\rput[bl](2.9,0.6){$a_{m-k+1}$}
\psdots[dotstyle=*,linecolor=blue](4.3,2)
\rput[bl](4.5,1.8){\blue{$a_n$}}
\psdots[dotstyle=*,linecolor=blue](5,2.7)
\rput[bl](5.2,2.4){\blue{$a_{n-1}$}}
\psdots[dotstyle=*,linecolor=blue](4.3,0)
\rput[bl](4.4,0){\blue{$a_2$}}
\psdots[dotstyle=*,linecolor=blue](4.9,-0.8)
\rput[bl](5,-0.7){\blue{$a_3$}}
\psdots[dotstyle=*,linecolor=blue](8,0.4)
\rput[bl](7.5,0.3){\blue{$a_i$}}
\psdots[dotstyle=*,linecolor=blue](8,1.5)
\rput[bl](7.3,1.4){\blue{$a_{i+1}$}}
\psdots[dotstyle=*,linecolor=blue](0.5,1)
\rput[bl](0.4,1.2){\blue{$t_j$}}
\psdots[dotstyle=*,linecolor=blue](1.5,1)
\rput[bl](1.4,1.2){\blue{$t_{j+1}$}}
\rput[bl](0.4,0.6){$a_{m-k+1}$}
\rput[bl](3.8,1.6){$t_0$}
\rput[bl](4.2,2.5){$t_0$}
\rput[bl](4.2,-0.7){$t_0$}
\rput[bl](3.8,0.2){$t_0$}
\rput[bl](8.2,0.9){$t_0$}
\end{pspicture*}
    \end{minipage}\hfill
\vspace{-6.5cm}
\caption{LOG corresponding to $\widehat{H}_n(m,k)$.\label{fig:CHRLOG}}
  \end{center}
\normalsize
\end{figure}

In fact, in Figure~\ref{fig:CHRLOG} the tail may be extended to contain $K+1$ vertices for any $K\equiv k$~mod~$n$; furthermore (and as expected) restricting to $k=1$ gives the LOG in~\cite{GilbertHowie} for the natural HNN extension $\widehat{H}(n,m)$ of $H(n,m)$.

There are other examples of cyclically presented groups considered in the literature whose natural HNN extensions are LOG groups. For instance, in~\cite[Example~3(4)]{SV} it was shown that the generalised Sieradski groups $S(r,n)=G_n((x_1x_3\ldots x_{2r-1})(x_2x_4\ldots x_{2r-2})^{-1})$ of~\cite{CHK} are such examples. We close this section by using the presentation of $\mathcal{L}_n(p_0,p_1,\ldots ,p_{r-1}) $ obtained in Theorem~\ref{thm:newSVpres} to exhibit further examples.\\

\begin{example}

\noindent \em The idea used in Corollary~\ref{cor:CHRLOG} can be extended to show that the cyclic presentations
\begin{alignat}{1}
\mathcal{G}_n(y_1^{-1}y_{1+l_1}y_{1+l_2}^{-1}y_{1+l_3}y_{1+l_4}^{-1}y_{1+l_5}\ldots y_{1+l_{s-2}}^{-1}y_{1+l_{s-1}}y_{1+l_s}^{-1})\label{eq:altsigncycypres}
\end{alignat}
($1\leq l_i\leq n$ for each $1\leq i\leq s$) belong to the family $\mathcal{L}_n(p_0,p_1,\ldots ,p_{r-1}) $.

\begin{itemize}
  \item[1.] The groups $G_1(n),H_1(n)$ from the ``certain cyclically presented groups'' series of papers, started in~\cite{JKO} have presentations of the form~(\ref{eq:altsigncycypres}):
      \begin{alignat*}{1}
        G_1(n)&=G_n(x_1^{-1}x_2x_1^{-1}x_2x_nx_1^{-1}x_n)\\
        &=G_n(x_1^{-1}x_nx_1^{-1}x_nx_{n-1}^{-1}x_nx_{n-1}^{-1}),\\
        H_1(n)&=G_n(x_1^{-1}x_2x_1^{-1}x_2x_1^{-2}x_nx_1^{-1}x_n)\\
        &= G_n( x_1^{-1}x_nx_1^{-1}x_nx_1^{-1}x_2x_1^{-1}x_2x_1^{-1}).
      \end{alignat*}

  \item[2.] The groups $G(l,n)$ of~\cite{YangkokKim} also have presentations of the form~(\ref{eq:altsigncycypres}). When $l\geq 2$ is even
     \begin{alignat*}{1}
G(l,n)&=G_n(x_1^{-1}(x_2x_1^{-1})^{(l-2)/2}x_2(x_nx_1^{-1})^{(l-2)/2}x_n)\\
&=G_n((x_1^{-1}x_2)^{(l-2)/2}x_1^{-1}x_2(x_3^{-1}x_2)^{(l-2)/2}x_3^{-1}),\\
\intertext{and when $l\geq 3$ is odd}
G(l,n)&=G_n(x_1(x_2^{-1}x_1)^{(l-1)/2}(x_n^{-1}x_1)^{(l-1)/2})\\
&= G_n( (x_1^{-1}x_n)^{(l-1)/2} (x_1^{-1}x_2)^{(l-1)/2}x_1^{-1}).
      \end{alignat*}
\end{itemize}\em
\end{example}

\section{Cyclically presented groups and their natural HNN extensions}

Presentations of the form
\[ \pres{a,c}{U(a,c),[a,c^n]}\]
in which $a$ appears in $U(a,c)$ with exponent sum $+1$ arise as presentations of natural HNN extensions of cyclically presented groups with admissible defining words (see Section~\ref{sec:prelims}) and as LOG presentations of the LOGs considered in Section~\ref{sec:LOG}. By applying a change of generator of the form $a\mapsto ac^\gamma$ we may further assume that $c$ appears in $U(a,c)$ with exponent sum zero. The following theorem examines these presentations and the groups they define in more detail.

\begin{theorem}\label{thm:exactseq}
Let $X$ be the presentation \(\pres{a,c}{U(a,c),[a,c^n]},\)
where
$U(a,c)= a^{\alpha_1}c^{\gamma_1}\ldots a^{\alpha_k}c^{\gamma_k}$ with $k\geq 1$ and $\sum_{i=1}^k \alpha_i =1$ and $\sum_{i=1}^k \gamma_i =0$. Then
\begin{enumerate}
 \item[(a)] the Hurewicz map $\pi_2(X)\rightarrow H_2(X)$ is surjective and hence the Schur multiplier $H_2(\pi_1(X))$ is trivial;
\item[(b)] the commutator subgroup $\pi_1(X)'$ has a cyclic presentation $\mathcal{G}_n(w)$ where
\[ w(x_1,\ldots,x_n)= x_1^{\alpha_1}x_{1+\gamma_1}^{\alpha_2}x_{1+\gamma_1+\gamma_2}^{\alpha_3}\ldots x_{1+\gamma_1+\cdots+\gamma_{k-1}}^{\alpha_k}\]
(which is admissible) and $\pi_1(X)\cong \widehat{G}_n(w)$;
\item[(c)] there is a short exact sequence
\begin{alignat}{1}
\{0\}\to\Z\pi_1(X)\to\pi_2(X)\to\pi_2(\mathcal{G}_n(w))\to\{0\}\label{eq:ses0}
\end{alignat}
of $\Z\pi_1(X)$-modules that splits as $\Z\pi_1(X)'$-modules;

\item[(d)] the sequence splits as a sequence of $\Z\pi_1(X)$-modules only in the trivial case where $\pi_2(\mathcal{G}_n(w))=\{0\}$.
\end{enumerate}

\end{theorem}

\begin{proof}
The cellular chain complex of $X$ implies that $H_2(X)$ is infinite cyclic, generated by the class of the relator $[a,c^n]$ and $H_1(X)$ is infinite cyclic generated by the class of the generator $c$. In the 1-relator group $\pres{a,c}{[a,c^n]}$ the element $c^n$ is central and so $[U,c^n]=1$ is a relator; therefore there is a planar van Kampen diagram $\Delta_0$ over the $1$-relator presentation $\pres{a,c}{[a,c^n]}$  with boundary label $[U,c^n]$. Since $a$ appears in $U$ with exponent-sum $1$, the $2$-cell $[a,c^n]$ appears in $\Delta_0$ with algebraic sum $1$.  In other words, $[a,c^n]^{+1}$ appears precisely once more than $[a,c^n]^{-1}$. (Since $1$-relator presentations of torsion-free groups are aspherical~\cite{Lyndon50} the algebraic sum is independent of the choice of van Kampen diagram.) Identify the sub-arcs of $\partial\Delta_0$ labelled by $c^n$ and $c^{-n}$ to form an annular diagram $\Delta_1$ with boundary labels $U$ and $U^{-1}$ and attach $2$-cells $U^{\pm 1}$ to $\partial\Delta_1$ to form a spherical diagram $\Delta$.  Now the relators $U(a,c)$ and $[a,c^n]$ appear in $\Delta$ with algebraic sums $0,1$ respectively, so $\Delta\in\pi_2(X)$ maps to a generator of $H_2(X)$ under the Hurewicz map $\pi_2(X)\rightarrow H_2(X)$. This map is therefore surjective so its cokernel, which is the Schur multiplier $H_2(\pi_1(X))$, is therefore trivial, proving part~(a).

Now use the spherical van Kampen diagram $\Delta$ to attach a $3$-cell to $X$ to form a $3$-complex $Y$ and let $\phi:\Z\pi_1(X)\to\pi_2(X)$ be the map of $\Z\pi_1(X)$-modules which sends a generator of $\Z\pi_1(X)$ to the class of $\Delta$. Then in the exact sequence
\begin{alignat}{1}
 0\rightarrow \mathrm{ker}(\phi) \rightarrow \Z\pi_1(X) \stackrel{\phi}{\rightarrow} \pi_2(X) \rightarrow \mathrm{coker}(\phi) \rightarrow 0 \label{eq:ses1}
\end{alignat}
of $\Z\pi_1(X)$-modules we have that $\mathrm{coker}(\phi)\cong \pi_2(Y)$~\cite{Whitehead41}.

Let $\tilde{X}$ be the universal cover of $X$. Then the cellular chain complex of $\tilde{X}$ implies that $H_2(\tilde{X})$ is the kernel of the second boundary map. Now $\pi_2(X)\cong\pi_2(\tilde{X})$ (see, for example,~\cite[Theorem~1.5, page 50]{Hilton}) and by the Hurewicz isomorphism theorem $\pi_2(\tilde{X})\cong H_2(\tilde{X})\subseteq C_2(\tilde{X})$ so $\pi_2(X)$ is a submodule of the free $\Z$-module on the $2$-cells of $\tilde{X}$.
Since $\pi_1(X)$ acts freely on $\tilde{X}$, it freely permutes the $2$-cells of $\tilde{X}$.  These form a $\mathbb{Z}$-basis for $C_2(\tilde{X})$, so $C_2(\tilde{X})$ is naturally identified with the free $\mathbb{Z}\pi_1(X)$-module on the set of $2$-cells of $X$. Under this identification the coefficient of the $2$-cell $U$ in the class of $\Delta$ has the form $g(1-c^n)$ for some $g\in\pi_1(X)$, since the $2$-cell $U$ appears precisely twice in $\Delta$, with opposite orientations, and these are separated by a path labelled $c^n$.
But by abelianizing we can see that $c$ has infinite order in $\pi_1(X)$, and so $1-c^n$ is not a zero-divisor in $\mathbb{Z}\pi_1(X)$.
Thus $\alpha U\neq 0$ for each $\alpha\in\mathbb{Z}\pi_1(X)\setminus\{0\}$. Hence $\Delta$ freely generates a free $\Z\pi_1(X)$-submodule of $\pi_2(X)$, and so $\phi$ is injective. Therefore the short exact sequence~(\ref{eq:ses1}) of $\Z\pi_1(X)$-modules becomes
\begin{alignat}{1}
 0\rightarrow \Z\pi_1(X) \stackrel{\phi}{\rightarrow} \pi_2(X) \rightarrow \pi_2(Y) \rightarrow 0. \label{eq:ses2}
\end{alignat}
Now $H_1(X)=H_1(Y)$ and $H_1(X)\cong \Z$ so the maximal abelian covers $\bar{X},\bar{Y}$ of $X,Y$ are regular covers with deck-transformation group $\Z$, generated by $c$. We can construct them as follows.  The set of $0$-cells is identified with $\Z$.  Corresponding to the $1$-cells $a,c$ of $X$ are $1$-cells $x_j,c_j$ ($j\in\Z$) of $\bar{X}$, where $x_j$ is a loop at $j$ and $c_j$ joins $j$ to $j+1$.  Corresponding to the $2$-cell $[a,c^n]$ in $X$ are $2$-cells $v_j$ in $\bar{X}$ ($j\in\Z$) with boundary paths $x_jc_jc_{j+1}\ldots c_{j+n-1}x_{j+n}^{-1}c_{j+n-1}^{-1}\ldots c_j^{-1}$; similarly, corresponding to the $2$-cell $U$ of $X$ is a set of $2$-cells $u_j$ ($j\in\Z$) in $\bar{X}$, where $u_{j+1}$ is the translate of $u_j$ by the natural generator for the deck-transformation group. Finally, corresponding to the $3$-cell $\Delta$ of $Y$ there are $3$-cells $\Delta_j$ of
$\bar{Y}$ ($j\in\Z$).  The $3$-cell $\Delta_j$ involves one appearance of each of the $2$-cells $u_j$ and $u_{j+n}$, and no other $2$-cell of the form $u_i$. It follows that $\bar{Y}$ homotopy retracts onto the $2$-subcomplex $W$ of $\bar{X}$ consisting of the $1$-skeleton $\bar{X}^{(1)}$ together with the $2$-cells $v_j$ ($j\in\Z$) and $u_1,\dots,u_n$.

Now the $1$-cells $c_j$ form the unique maximal tree in $W$.  Collapsing this tree to a point
gives the following presentation of the commutator subgroup $\pi_1(X)'$ which is homotopy-retract to $W$
$$\pres{x_j\ (j\in\Z)}{x_jx_{j+n}^{-1} (j\in\Z), w_1,\dots,w_n}$$
where
$$w_j= x_j^{\alpha_1}x_{j+\gamma_1}^{\alpha_2}x_{j+\gamma_1+\gamma_2}^{\alpha_3}\ldots x_{j+\gamma_1+\cdots+\gamma_{k-1}}^{\alpha_k}.$$
One final homotopy retraction uses the relators $x_jx_{j+n}^{-1}$ to restrict the generating set to $\{x_1,\dots,x_n\}$ and so we obtain the cyclic presentation $Z=\mathcal{G}_n(w)$ of part~(b), and it is clear that $\pi_1(X)$ is the natural HNN extension of $\pi_1(X)'$.

Now $\pi_2(Y)\cong \pi_2(\bar{Y})$ so we have $\pi_2(Y)\cong\pi_2(\bar{Y})\cong\pi_2(W)\cong\pi_2(Z)$ so in~(\ref{eq:ses2}) we may replace $\pi_2(Y)$ by $\pi_2(Z)$ to get the short exact sequence~(\ref{eq:ses0}) of $\Z\pi_1(X)$-modules of part~(c). The homotopies and inclusions
$ Z \simeq W\subset \bar{X}\subset \bar{Y}\simeq Z $
together with the identity map $Z\rightarrow Z$ induces a sequence of $\Z\pi_1(X)'$-modules
$ \pi_2(Z) \cong \pi_2(W) \rightarrow \pi_2(\bar{X})\rightarrow \pi_2(\bar{Y}) \cong \pi_2(Z).$
But $\pi_2(\bar{X})\cong \pi_2(X)$ so we get a sequence of $\Z\pi_1(X)'$-modules
$ \pi_2(Z)\rightarrow \pi_2(X) \rightarrow \pi_2(Z) $
which shows that the sequence~(\ref{eq:ses0}), when regarded as a sequence of $\Z\pi_1(X)'$-modules, splits, proving part~(c).

To prove~(d), consider a map $s:\pi_2(Z)\rightarrow \pi_2(X)$ of $\Z \pi_1(X)$-modules. Now $c^n$ acts trivially on $Z$ and hence on $\pi_2(Z)$ so $1-c^n$ annihilates $\pi_2(Z)$ and so annihilates $s(\pi_2(Z))\subseteq \pi_2(X) =\pi_2(\tilde{X})$. The Hurewicz isomorphism theorem implies that $\pi_2(\tilde{X})\cong H_2(\tilde{X})\subseteq C_2(\tilde{X})$ so $1-c^n$ annihilates a submodule of a free $\Z\pi_1(X)$-module (namely $s(\pi_2(Z))$). But as pointed out above $1-c^n$ is not a zero-divisor in $\mathbb{Z}\pi_1(X)$, so $s(\pi_2(Z))=0$ and the sequence~(\ref{eq:ses0}) of $\Z\pi_1(X)$-modules does not split when $\pi_2(Z)\neq\{0\}$.
\end{proof}

Theorem~1.1 of~\cite{GilbertHowie} (and its corollary), which concerns the presentations $\mathcal{H}(n,m)$, can be obtained as a corollary to Theorem~\ref{thm:exactseq}.

\section{The conjugacy problem for natural HNN extensions}\label{sec:conjugacy}

In Section~\ref{sec:smallcanc} we will give conditions under which the conjugacy problem for $H_n(m,k)$ is solvable. In this section, following~\cite{BMV}, we study the conjugacy problem in the natural HNN extension $\widehat{G}_n(w)=G_n(w)\rtimes \Z$ via the twisted conjugacy problem in $G_n(w)$. In Corollary~\ref{cor:consequencesofsmallcanc} we will apply our result to $\widehat{H}_n(m,k)$.

Given a group $G$ and an automorphism $\phi\in \mathrm{Aut}(G)$, an element $u\in G$ is said to be \em $\phi$-twisted conjugate \em to $v\in  G$, denoted $u\sim_\phi v$, if there exists $g\in G$ such that $g^{-1}u\phi(g)=v$. The \em $\phi$-twisted conjugacy problem \em in $G$
is solvable if it is algorithmically decidable if $u\sim_\phi v$. In particular, the conjugacy problem is the id-twisted conjugacy problem.
The \em twisted conjugacy problem \em in $G$ is solvable if the \em $\phi$-twisted conjugacy problem \em in $G$
is solvable for all $\phi\in\mathrm{Aut}(G)$.

\begin{theorem}\label{thm:twistedconj}
If $G$ is a finitely generated hyperbolic group and $\phi\in\mathrm{Aut}(G)$ has finite order, then the $\phi$-twisted conjugacy problem in $G$ is solvable.
\end{theorem}

\begin{proof}
The proof is analogous to the proof that the conjugacy problem is solvable for hyperbolic groups.  Since $\phi$ has finite order and $G$ is finitely generated, we can choose a finite generating set $X$ for $G$ such that $X^{-1}=X=\phi(X)$.  Let $\delta$ be the hyperbolicity constant for the Cayley graph $\Gamma=\Gamma(G,X)$.  Then geodesic triangles in $\Gamma$ are $\delta$-slim, so geodesic quadrilaterals are $2\delta$-slim.  Suppose that $u\phi(g)=gv$ in $G$.  Choose geodesic segments $\gamma$ from $1$ to $g$, $\alpha$ from $1$ to $u$ and $\beta$ from $g$ to $gv$.  Applying $\phi$ to (the label of) each edge of $\gamma$ gives a geodesic segment $\phi(\gamma)$ from $1$ to $\phi(g)$.  Then we can construct a geodesic quadrilateral with vertices $1,u,g,gv=u\phi(g)$ and edges $\alpha,\beta,\gamma$ and $\gamma'=u\phi(\gamma)$.

Note that $\gamma,\gamma'$ have the same length, $L$ say, by construction.  Moreover, $\gamma'(n)=u\phi(\gamma(n))$ for each $1\leq n\leq L$. For each $n$, $\gamma(n)$ is $2\delta$-close to a vertex of $\alpha$, $\beta$ or $\gamma'$.
When $||u||+2\delta<n<L-(||v||+2\delta)$ only the last of these is possible, by the triangle inequality.  Thus $d(\gamma(n),\gamma'(n'))\leq 2\delta$ for some $n'$.
Consider the geodesic quadrilateral with vertices $1,\gamma(n),\gamma'(n'),u$. We have
\[n= d(1,\gamma(n))\leq d(1,u)+d(u,\gamma'(n'))+d(\gamma'(n'),\gamma(n))\leq ||u||+n'+2\delta\]
and similarly $n'\leq ||u||+n+2\delta$,
so $|n-n'|\leq ||u||+2\delta$. Hence $d(\gamma(n),\gamma'(n))\leq 2\delta +|n-n'|\leq ||u||+4\delta$.

The number of elements of $G$ of length at most $||u||+4\delta$ is less than $|X|^{||u||+4\delta+1}$.  Hence if $L>N$, where $N=|X|^{||u||+4\delta+1}+2\mu+4\delta$, the number of $n$'s in the given range exceeds the number of elements of $G$ of length at most $||u||+4\delta$. Therefore there exist $n_0<n_1$ with $\gamma(n_0)^{-1}\gamma'(n_0)=\gamma(n_1)^{-1}\gamma'(n_1)$. Setting $h=\gamma(n_0)\gamma(n_1)^{-1}g$ gives $||h||\leq L+n_0-n_1<L$ and
\begin{alignat*}{1}
u\phi(h)&=u\phi(\gamma(n_0))\phi(\gamma(n_1)^{-1})\phi(g)\displaybreak[1]\\
&=\gamma'(n_0)\gamma'(n_1)^{-1}u\phi(g)\displaybreak[1]\\
&=\gamma(n_0)\gamma(n_1)^{-1}gv\\
&=hv.
\end{alignat*}
Repeating the procedure gives a solution to the equation $u\phi(g)=gv$  with $||g||\leq N$. Since $N$ is a computable function of $u,v$ we can apply the solution of the word problem for hyperbolic groups to test equality of the elements $u\phi(g)$ and $gv$ for each of the finitely many $g$ with $||g||<N$.
\end{proof}

\begin{corollary}\label{cor:conjprobcycpres}
If $G$ is hyperbolic and $\phi\in\mathrm{Aut}(G)$ has finite order then the conjugacy problem for $G\rtimes_\phi \Z$ is solvable. In particular, if a cyclically presented group $G_n(w)$ is hyperbolic then the conjugacy problem for its natural HNN extension $\widehat{G}_n(w)$ is solvable.
\end{corollary}

\begin{proof}
Let $t$ be a generator of $\Z$ so that $G\rtimes_\phi \Z$ has elements $gt^m$ with $g\in G$, $m\in \Z$ and multiplication
$(gt^m)(ht^l)=g\phi^m(h)t^{m+l}$.
Two elements $ut^p,vt^q\in G\rtimes_\phi \Z$ are conjugate if and only if there exist $gt^m\in G\rtimes_\phi \Z$ such that $(ut^p)(gt^m)=(gt^m)(vt^q)$; equivalently if and only if there exists $g\in G$, $m\in \Z$ such that $(u\phi^p(g))t^{p+m}=(g\phi^m(v))t^{m+q}$; equivalently if and only if $p=q$ and there exists $g\in G, m\in \Z$ such that $u\phi^p(g)=g\phi^m(v)$; equivalently if and only if $p=q$ and there exists $m\in \Z$ such that $u\sim_{\phi^p}\phi^m (v)$.

That is, the conjugacy problem is solvable in  $G\rtimes_\phi \Z$ if and only if for any $u,v\in G$ and $p\in \Z$ it is decidable if $u$ is $\phi^p$-twisted conjugate to any of the (finitely many) elements $\phi^m(v)\in G$. Since $\phi^p$ has finite order this is decidable by Theorem~\ref{thm:twistedconj}.
\end{proof}

\section{Small cancellation conditions}\label{sec:smallcanc}

In this section we consider small cancellation conditions for the presentations $\mathcal{H}_n(m,k)$. We refer the reader to~\cite[Chapter~V]{LyndonSchupp} for basic definitions regarding small cancellation theory. The strongest C($p$) condition that $\mathcal{H}_n(m,k)$ can satisfy is C($3$). This condition is most useful when it occurs in conjunction with T($6$) or T($7$). In Theorem~\ref{thm:CHRsmallcanc} we classify the C(3)-T(7) and C(3)-T(6) (special and non-special) presentations $\mathcal{H}_n(m,k)$. (We will define the term \em special \em shortly, but for the moment we note that they are examples of \em triangular presentations \em -- presentations in which each relator has length~3.)

To do this we use the characterization~(\cite{HillPrideVella}) of the C($3$)-T($q$) conditions in terms of the \em star graph \em of a presentation $\pres{X}{R}$. (The star graph is also called the \em co-initial graph \em or \em star complex \em or \em Whitehead graph\em.)  This is the graph $\Gamma$ with vertex set $X \cup X^{-1}$ and with an edge from $x$ to $y$ for each distinct word of the form $x^{-1}yu$ ($x\neq y$) that is a cyclic permutation of a relator or its inverse~\cite[page~61]{LyndonSchupp}. (The edges occur in inverse pairs.) By~\cite{HillPrideVella} a presentation $\pres{X}{R}$ in which each relator has length at least~3 satisfies C(3)-T($q$) ($q>4$) if and only if $\Gamma$ has no cycle of length less than~$q$. (See also~\cite{EdjvetHowie} for further explanation.) If a presentation satisfies T$(q)$, $q>4$ then each piece has length~1 (this is proved in~\cite[Lemma~1]{GerstenShortI}, where the observation is attributed to Pride). Using the terminology of~\cite{HowieSQ} a C(3)-T(6) presentation is \em special \em if every relator has length~3 and $\Gamma$ is isomorphic to the incidence graph of a finite projective plane.
By Lemma~1.3.6 of~\cite{Maldeghem} this condition on $\Gamma$ is equivalent to $\Gamma$ being connected and bipartite of diameter~3 and girth~6, with each vertex having degree at least~3.

\begin{theorem}\label{thm:CHRsmallcanc}
Let $A=k$, $B=k-m$.
\begin{itemize}
\item[(a)] The presentation $\mathcal{H}_n(m,k)$ is a C(3)-T(6) presentation if and only if
        $tA \not \equiv 0, tB \not \equiv 0$ ($1\leq t \leq 5$) and $A\not \equiv \pm B$, $A\not \equiv \pm 2B$, $B\not \equiv \pm 2A$ (all mod~$n$), in which case it is a non-special C(3)-T(6) presentation.
\item[(b)] The presentation $\mathcal{H}_n(m,k)$ is a C(3)-T(7) presentation if and only if
        $tA \not \equiv 0, tB \not \equiv 0$ ($1\leq t \leq 6$) and $A\not \equiv \pm 2B$, $A\not \equiv \pm 3B$, $B\not \equiv \pm 2A$, $B\not \equiv \pm 3A$, $2A\not \equiv \pm 2B$ (all mod~$n$).
\end{itemize}
\end{theorem}

\begin{proof}
The star graph $\Gamma$ has vertices $x_i,x_i^{-1}$ and edges $\{ x_i,x_{i+m}^{-1} \}$, $\{x_i ,x_{i+B} \}$, $\{x_i^{-1} ,x_{i+A}^{-1} \}$ for each $1\leq i \leq n$ (subscripts mod~$n$), which we will refer to as edges of type X,Y,Z, respectively. Thus $\Gamma$ has a $t$-cycle involving only $x_i$ vertices if and only if $tB\equiv 0$~mod~$n$ and has a $t$-cycle involving only $x_i^{-1}$ vertices if and only if $tA\equiv 0$~mod~$n$. Consider then cycles involving both $x_i$ and $x_i^{-1}$ vertices. Clearly these must involve an even number of X edges, and since an X edge cannot be followed immediately by another X edge there are no such 2-cycles or 3-cycles. We are therefore interested in the 4,5, and 6-cycles.

Any 4-cycle must be of the form XZXY, any 5-cycle of the form XZXYY, or XYXZZ, and any 6-cycle of the form XZXYYY, XYXZZZ, or XZZXYY. The paths of the form XZXY are $x_i - x_{i+m}^{-1} - x_{i+m\pm A}^{-1} - x_{i\pm A} - x_{i\pm A\pm B}$, which are cycles if and only if $A\pm B \equiv 0$~mod~$n$. The paths of the form XZXYY are $x_i - x_{i+m}^{-1} - x_{i+m\pm A}^{-1} - x_{i\pm A} - x_{i\pm A\pm B}- x_{i\pm A\pm 2B}$, which are cycles if and only if $A\pm 2B \equiv 0$~mod~$n$. Similarly cycles of the form XYXZZ occur if and only if $B\pm 2A \equiv 0$~mod~$n$.
The paths of the form XZXYYY are $x_i - x_{i+m}^{-1} - x_{i+m\pm A}^{-1} - x_{i\pm A} - x_{i\pm A\pm B}- x_{i\pm A\pm 2B}- x_{i\pm A\pm 2B}$, which are cycles if and only if $A\pm 3B \equiv 0$~mod~$n$. Similarly cycles of the form XYXZZZ occur if and only if $B\pm 3A \equiv 0$~mod~$n$.
The paths of the form XZZXYY are $x_i - x_{i+m}^{-1} - x_{i+m\pm A}^{-1} - x_{i+m\pm 2A}^{-1} - x_{i\pm 2A} - x_{i\pm 2A\pm B}- x_{i\pm 2A\pm 2B}$, which are cycles if and only if $2A\pm 2B \equiv 0$~mod~$n$. (For use in the proof of Theorem~\ref{thm:H_n(3,1)} we remark here that 7-cycles are of the form XZXYYYY, XYXZZZZ, XZZXYYY, or XYYXZZZ and occur if and only if $A\pm 4B\equiv 0$~mod~$n$ or $B\pm 4A\equiv 0$~mod~$n$ or $2A\pm 3B\equiv 0$~mod~$n$ or $2B\pm 3A\equiv 0$~mod~$n$.)

It remains to show that the C(3)-T(6) presentations are non-special. Suppose for contradiction that $\Gamma$ is the incidence graph of a projective plane $P$. Then since the vertices of $\Gamma$ have degree~3 it follows that $P$ has order~2 so has 7 points and 7 lines (see for example~\cite[Section~16.7 and exercise 16.10.10]{Biggs}). Therefore $\Gamma$ has precisely 14 vertices so $n=7$, but the congruences imply that $\mathcal{H}_n(m,k)$ is not a C(3)-T(6) presentation in this case, a contradiction.
\end{proof}

Some of the excluded cases correspond to known groups. By~\cite[Section~1]{BV} if $A\equiv 0$ or $B\equiv 0$~mod~$n$ then $G_n(m,k)$ is the trivial group; if $A\equiv B$~mod~$n$ then $G_n(m,k)$ is the free product of $(k,n)$ copies of $\Z_{2^n-1}$; if $A\equiv -B$ then $G_n(m,k)$ is the free product of $(k,n)$ copies of $S(2,n)$; if $A\equiv 2B$ or $B\equiv 2A$ then $G_n(m,k)$ is the free product of $(m,n)$ or $(k,n)$ copies of $F(2,n)$, respectively. If $A\equiv -2B$ or $B\equiv -2A$ then $G_n(m,k)$ is the free product of $(m,n)$ or $(k-m,n)$ copies of $H_n(3,1)$, respectively; we study these groups further in Section~\ref{sec:hyperbolicity}.

In Corollary~\ref{cor:consequencesofsmallcanc} we observe some consequences. A countable group $G$ is said to be \em SQ-universal \em if every countable group can be embedded in a quotient group of~$G$; in particular, SQ-universal groups contain a free subgroup of rank~2. The dependence problems~\cite{PrideGlasgow} DP($n$), $n\geq 1$, generalize the word problem (DP($1$)) and the conjugacy problem (DP($2$)). A group is \em Hopfian \em if it is not isomorphic to any of its proper quotients.

\begin{corollary}\label{cor:consequencesofsmallcanc}
\begin{itemize}
  \item[(a)] Suppose the conditions of Theorem~\ref{thm:CHRsmallcanc}(a) hold. Then the group $H_n(m,k)$ has solvable word and conjugacy problems and is automatic, SQ-universal, torsion-free, and acts properly and cocompactly on a CAT$(0)$ space.
  \item[(b)] Suppose the conditions of Theorem~\ref{thm:CHRsmallcanc}(b) hold. Then $H_n(m,k)$ has solvable dependence problems DP($n$) for all $n\geq 1$ and is non-elementary hyperbolic, Hopfian, and acts properly and cocompactly on a CAT$(-1)$ space, and the class of groups $H_n(m,k)$ has solvable isomorphism problem. Further, the natural HNN extension $\widehat{H}_n(m,k)$ has solvable conjugacy problem
\end{itemize}
\end{corollary}

\begin{proof}
C(3)-T(6) groups have solvable word and conjugacy problems, by Theorems~6.3 and~7.6 of~\cite[Chapter~V]{LyndonSchupp} (see also ~\cite[Section~3]{ElMosalamyPride}) and they are automatic by~\cite[Theorem~2]{GerstenShortII}. Groups defined by non-special C(3)-T(6) groups are SQ-universal~\cite{HowieSQ}. The Cayley complex of a C(3)-T(6) (resp. C(3)-T(7)) presentation in which each relator has length at least~3 is a CAT$(0)$ (resp. CAT$(-1)$) space (see for example~\cite[Proposition~5.25]{BridsonHaefliger}), and hence contractible. Since no relator is a proper power the Cayley complex coincides with the universal cover of the presentation complex, upon which $H_n(m,k)$ acts freely. Thus $H_n(m,k)$ acts freely on a finite-dimensional contractible complex, so is torsion-free.

C(3)-T(7) groups have solvable dependence problems DP($n$) for all $n\geq 1$~\cite[Theorems~2 and~4]{PrideGlasgow} and are hyperbolic~\cite[Corollary~3.3]{GerstenShortI} (see also~\cite[Theorem~4]{PrideGlasgow}); the class of hyperbolic groups has solvable isomorphism problem~\cite{DahmaniGuirardel}; torsion-free hyperbolic groups are Hopfian by~\cite{Sela}. Since $H_n(m,k)$ is hyperbolic, $\widehat{H}_n(m,k)$ has solvable conjugacy problem by Corollary~\ref{cor:conjprobcycpres}.
\end{proof}

We note a corollary which shows that nearly all the groups $H_p(m,k)$ with $p$ an odd prime are C(3)-T(7). By~\cite[Proposition~7]{COS} $H_p(m,k)$ is isomorphic to one of $H_p(1,k')$ for some $1\leq k' \leq (p+1)/2$. The cases $k'=1,2$ and $(p+1)/2$ give the trivial group, the Fibonacci group $F(2,p)$, and the Sieradski group $S(2,p)$, respectively. For the remaining cases, for each $p$, the next corollary shows that precisely one of the presentations is not C(3)-T(7) (in which case it is not C(3)-T(6)). Note that this is not the case when $n$ is composite, for example by~\cite[Table~1]{COS} there are precisely 5 non-isomorphic (non-trivial) groups $H_n(m,k)$ when $n=6$, namely $F(2,6)=H_6(1,2),\ S(2,6)\cong H_6(4,5),\ \Z_7\cong H_6(1,3),\ \Z_9\cong H_6(2,3),\ \Z_2^3 \rtimes \Z_7 \cong H_6(3,4)$.

\begin{corollary}\label{cor:prime}
Let $p$ be an odd prime and let $3\leq k \leq (p-1)/2$ be an integer. The presentation $\mathcal{H}_p(1,k)$ is C(3)-T(7) if and only if $k\neq (p + 1)/3$ (when $p\equiv 2$~mod~$3$) or $k\neq (p + 2)/3$ (when $p\equiv 1$~mod~$3$).
\end{corollary}

\begin{proof}
This follows by testing the congruences. The case $k=(p+1)/3$ corresponds to $B\equiv -2A$~mod~$n$ and the case $k=(p+2)/3$ corresponds to $A\equiv -2B$~mod~$n$.
\end{proof}

\section{The groups $H(n,3)$}\label{sec:hyperbolicity}

As observed in Section~\ref{sec:smallcanc}, the presentations $\mathcal{H}_n(3,1)$, which are the presentations $\mathcal{H}(n,3)$ of~\cite{GilbertHowie}, are not C(3)-T(6) presentations; we show here however that, in most cases, $H(n,3)=H_n(3,1)$ is hyperbolic and SQ-universal.

The group $H(9,3)$ was proved to be infinite in~\cite[Lemma~15]{COS}, using Newman's variant of the Golod-Shafarevich inequality on an index~448 subgroup. Further, $H(8,3)$ is finite of order $3^{10}\cdot 5$ (by a calculation in GAP~\cite{GAP}), $H(7,3)$ is infinite (this is attributed to Thomas in~\cite{GilbertHowie}), $H(6,3)\cong \Z_2^3\rtimes \Z_7$, $H(5,3)\cong\Z_{11}$, $H(4,3)\cong \Z_5$.
The groups $H(10,3),H(12,3),H(14,3)$ are infinite by~\cite{GilbertHowie}.

\begin{theorem}\label{thm:H_n(3,1)}
Let $n\geq 11$, $n\neq 12,14$. Then $H(n,3)$ is hyperbolic.
\end{theorem}

\begin{proof}
To prove that $H(n,3)$ is hyperbolic it suffices to show that it has a linear isoperimetric function (see for example~\cite[Theorem~3.1]{Gersten}). That is, we must show that there is a linear function $f:\mathbb{N}\rightarrow \mathbb{N}$ such that for all $N\in \mathbb{N}$ and all freely reduced words $w\in F_n$ with length at most $N$ that represent the identity of $H(n,3)$ we have $\mathrm{Area}(w)\leq f(N)$ (where $\mathrm{Area}(w)$ denotes the minimum number of 2-cells in a reduced van Kampen diagram over $\mathcal{H}(n,3)$ with boundary label $w$). Without loss of generality we may assume that the boundary of such a van Kampen diagram is a simple closed curve. Note that each 2-cell in a van Kampen diagram over $\mathcal{H}(n,3)$ is a triangle.

By the proof of Theorem~\ref{thm:CHRsmallcanc} cycles in the star graph of $\mathcal{H}(n,3)$ have length~5 or at least~8. The star graph has edges $\{x_i,x_{i+3}^{-1}\}, \{x_i,x_{i+2}\}, \{x_i^{-1},x_{i+1}^{-1}\}$ for each $1\leq i \leq n$ (subscripts mod~$n$), which we again refer to as edges of type X,Y,Z, respectively. Any 5-cycle must be of the form XZXYY or XZZXY; there are no cycles of the first form and cycles of the second form are (i)~$x_i-x_{i+3}^{-1}-x_{i+2}^{-1}-x_{i+1}^{-1}-x_{i-2}-x_i$. Similarly, cycles of length 8 are of the form (ii)~$x_i-x_{i+2}-x_{i+4}-x_{i+6}-x_{i+8}-x_{i+10}-x_{i+12}-x_{i+14}-x_{i}$ (when $n=16$) or (iii)~$x_i-x_{i+3}^{-1}-x_{i+4}^{-1}-x_{i+1}-x_{i+3}-x_{i+6}^{-1}-x_{i+5}^{-1}-x_{i+2}-x_{i}$ or (iv)~$x_i^{-1}-x_{i-3}-x_{i-1}-x_{i+1}-x_{i+4}^{-1}-x_{i+3}^{-1}-x_{i+2}^{-1}-x_{i+1}^{-1}-x_{i}^{-1}$.

Consider a reduced van Kampen diagram over $\mathcal{H}(n,3)$ whose boundary is a simple closed curve.
Given a degree 5 interior vertex corresponding to (i), completing the incident 2-cells shows that at most 2 of the adjacent vertices can be interior of degree 5; those 2 vertices are non-adjacent so there are no triangles whose vertices are all interior of degree 5.
For a degree 8 interior vertex corresponding to~(ii), completing the incident 2-cells shows that at most 4 of the adjacent vertices can be interior of degree 5.
For a vertex corresponding to~(iii), completing the incident 2-cells reveals that three of the adjacent vertices are not of the form of an interior  degree 5 vertex.
For a vertex corresponding to~(iv), completing the incident  2-cells reveals that one of the adjacent vertices is not of the form of an interior  degree 5 vertex. Of the remaining 7 adjacent vertices completing the incident 2-cells shows that at most 5 of the adjacent vertices can be interior of degree 5.
In summary, an interior vertex of degree 5 has at most 2 neighbours that are interior of degree 5 and any interior vertex of degree 8 has at least three neighbouring vertices that are not interior of degree~5.

Consider now degree 9 interior vertices; we claim that any such vertex has a neighbouring vertex that is not interior of degree 5. A vertex corresponding to a cycle containing a subpath of the form ZZZ has an adjacent vertex that is not of the form of an interior degree 5 vertex. A vertex corresponding to a cycle containing a subpath of the form YYYY has at least one vertex that is not interior of degree 5. Therefore the only degree 9 vertices that can have all its neighbouring vertices interior of degree 5 correspond to cycles of the form
XZXYXZXYY or XYXZXYXZZ. Cycles of the first type are $x_i-x_{i+3}^{-1}-x_{i+4}^{-1}-x_{i+1}-x_{i+3}-x_{i+6}^{-1}-x_{i+7}^{-1}-x_{i+4}-x_{i+2}-x_i$, and a vertex corresponding to such a cycle has a neighbouring interior vertex that is not of degree 5; there are no cycles of the second form, so the claim is proved.

Bearing in mind the above information, we assign angles to corners of each triangle
$\Delta$ of the van Kampen diagram as follows.
 If no vertex of $\Delta$ is interior of index $5$, then
each corner has angle $59^\circ$; if precisely one vertex of $\Delta$ is interior of index $5$, then
the corner at that vertex has angle $80^\circ$ and
each of the other corners has angle $49^\circ$; if two vertices of $\Delta$ are interior of index $5$, then
each of the corners of $\Delta$ at those vertices each have angle $70^\circ$,
and the third corner of $\Delta$ has angle $39^\circ$.
(As observed above, it is impossible for a triangle $\Delta$ to have three
interior vertices of index $5$.)

It follows from our various observations that no interior vertex has angle sum less than $360^\circ$,
while no triangle has angle sum greater than $179^\circ$.  Hence if there are $T$ triangles and $B$
boundary vertices in the diagram, then
the sum of the angles at boundary vertices is at most $(180B-T-360)^\circ$.  However, this sum is at least
$39B^\circ$, since each corner has angle at least $39^\circ$.  Hence
$39B<(180B-T-360)$, so $T<141B-360$ and $f(N)=141N-360$ is a suitable linear isoperimetric function.
\end{proof}

It seems likely that more detailed analysis could be used to show that the result extends to the cases $n=12,14$.

\begin{corollary}\label{cor:H(n,3)SQ}
Let $n\geq 11$, $n\neq 12,14$. Then $H(n,3)$ is SQ-universal.
\end{corollary}

\begin{proof}
By~\cite[Theorem~3.2]{GilbertHowie} $\mathcal{H}(n,3)$ is aspherical, so (since there are no proper powered relators) the group $H(n,3)$ is torsion-free. Further, by~\cite[Theorem~2.3]{GilbertHowie} $\mathcal{H}(n,3)^{\mathrm{ab}}$ is finite and non-trivial, and hence $H(n,3)$ is infinite and non-abelian. Thus $H(n,3)$ is torsion-free and non-cyclic, so it is not virtually cyclic. Therefore it is non-elementary hyperbolic and torsion-free so it is SQ-universal by~\cite{Olshanskii} or~\cite{Delzant}.
\end{proof}

By Corollary~\ref{cor:conjprobcycpres} we have

\begin{corollary}\label{cor:H(n,3)conjugacy}
Let $n\geq 11$, $n\neq 12,14$. Then the natural HNN extension $\widehat{H}(n,3)$ has solvable conjugacy problem.
\end{corollary}


\begin{thebibliography}{10}

\bibitem{BV}
V.G. Bardakov and A.Yu. Vesnin.
\newblock A generalization of {F}ibonacci groups.
\newblock {\em Alg. and Logic}, 42:131--160, 2003.

\bibitem{Biggs}
N. L. Biggs.
\newblock {\em {Discrete mathematics}}.
\newblock Oxford University Press, second edition, 2002.

\bibitem{Bogley}
W. A. Bogley.
\newblock Local collapses for diagrammatic reducibility.
\newblock In {\em Topology and {C}ombinatorial {G}roup {T}heory}, Lect. Notes
  Math. 1440, pages 27--38, 1990.

\bibitem{BMV}
O. Bogopolski, A. Martino, and E. Ventura.
\newblock {Orbit decidability and the conjugacy problem for some extensions of
  groups.}
\newblock {\em Trans. Am. Math. Soc.}, 362(4):2003--2036, 2010.

\bibitem{BridsonHaefliger}
M. R. Bridson and A. Haefliger.
\newblock {\em {Metric spaces of non-positive curvature.}}
\newblock {Grundlehren der Mathematischen Wissenschaften. 319. Berlin:
  Springer}, 1999.

\bibitem{CHK}
A. Cavicchioli, F. Hegenbarth, and A.C. Kim.
\newblock {A geometric study of Sieradski groups.}
\newblock {\em Algebra Colloq.}, 5(2):203--217, 1998.

\bibitem{CHR}
A. Cavicchioli, F. Hegenbarth, and D. Repov\v{s}.
\newblock On manifold spines and cyclic presentations of groups.
\newblock In {\em Knot theory}, volume~42, pages 49--56. Warszawa: Polish
  Academy of Sciences, Institute of Mathematics, Banach Cent. Publ., 1998.

\bibitem{COS}
A. Cavicchioli, E.A. O'Brien, and F. Spaggiari.
\newblock {On some questions about a family of cyclically presented groups.}
\newblock {\em J. Algebra}, 320(11):4063--4072, 2008.

\bibitem{DahmaniGuirardel}
F. Dahmani and V. Guirardel.
\newblock {The isomorphism problem for all hyperbolic groups.}
\newblock {\em Geom. Funct. Anal.}, 21(2):223--300, 2011.

\bibitem{Delzant}
T. Delzant.
\newblock {Sous-groupes distingu\'es et quotients des groupes hyperboliques.}
\newblock {\em Duke Math. J.}, 83(3):661--682, 1996.

\bibitem{EdjvetHowie}
M. Edjvet and J. Howie.
\newblock Star graphs, projective planes and free subgroups in small
  cancellation groups.
\newblock {\em Proc. Lond. Math. Soc., III. Ser.}, 57(2):301--328, 1988.

\bibitem{ElMosalamyPride}
M. S. El-Mosalamy and S. J. Pride.
\newblock {On T(6)-groups.}
\newblock {\em Math. Proc. Camb. Philos. Soc.}, 102:443--451, 1987.

\bibitem{GAP}
The GAP~Group.
\newblock {\em {GAP -- Groups, Algorithms, and Programming, version 4.4.9}},
  2006.


\bibitem{GerstenShortI}
S.M. Gersten and H.B. Short.
\newblock {Small cancellation theory and automatic groups.}
\newblock {\em Invent. Math.}, 102(2):305--334, 1990.

\bibitem{GerstenShortII}
S.M. Gersten and H.B. Short.
\newblock {Small cancellation theory and automatic groups. II.}
\newblock {\em Invent. Math.}, 105(3):641--662, 1991.

\bibitem{Gersten}
S.M. Gersten.
\newblock {Isoperimetric and isodiametric functions of finite presentations.}
\newblock {G. A. Niblo et al. (ed.), Geometric group theory. Volume 1.
  Proceedings of the symposium held at the Sussex University, Brighton (UK),
  July 14-19, 1991. Cambridge University Press, Lond. Math. Soc.
  Lect. Note Ser. 181, 79-96}, 1993.

\bibitem{GilbertHowie}
N.D. Gilbert and J. Howie.
\newblock {LOG} groups and cyclically presented groups.
\newblock {\em J. Algebra}, 174(1):118--131, 1995.

%\bibitem{Gromov}
%M. Gromov.
%\newblock {Hyperbolic groups.}
%\newblock {Essays in group theory, Publ., Math. Sci. Res. Inst. 8, 75-263
%  (1987).}

\bibitem{HillPrideVella}
P. Hill, S.J. Pride, and A.D. Vella.
\newblock On the {T($q$)}-conditions of small cancellation theory.
\newblock {\em Isr. J. Math.}, 52:293--304, 1985.

\bibitem{Hilton}
P.J. Hilton.
\newblock {\em {An introduction to homotopy theory.}}
\newblock Number~43 in Cambridge Tracts in Mathematics and Mathematical
  Physics. Cambridge University Press, 1953.

\bibitem{HowieSQ}
J. Howie.
\newblock On the {SQ}-universality of {T(6)}-groups.
\newblock {\em Forum Math.}, 1(3):251--272, 1989.

\bibitem{JKO}
D.L. Johnson, A.C. Kim, and E.A. O'Brien.
\newblock {Certain cyclically presented groups are isomorphic.}
\newblock {\em Commun. Algebra}, 27(7):3531--3536, 1999.

\bibitem{JohnsonMawdesley}
D.L. Johnson and H. Mawdesley.
\newblock Some groups of {F}ibonacci type.
\newblock {\em J. Aust. Math. Soc.}, 20:199--204, 1975.

\bibitem{YangkokKim}
Y. Kim.
\newblock {About some infinite family of 2-bridge knots and 3-manifolds.}
\newblock {\em Int. J. Math. Math. Sci.}, 24(2):95--108, 2000.

\bibitem{Lyndon50}
R.C. Lyndon.
\newblock {Cohomology theory of groups with a single defining relation.}
\newblock {\em Annals of Math.}, 52(3):650--665, 1950.

\bibitem{LyndonSchupp}
R. C. Lyndon and P. E. Schupp.
\newblock {\em Combinatorial group theory}, volume~89 of {\em Ergebnisse der
  Mathematik und ihrer Grenzgebiete}.
\newblock Springer-Verlag, 1977.

%\bibitem{Odoni}
%R.W.K. Odoni.
%\newblock Some {D}iophantine problems arising from the theory of
%  cyclically-presented groups.
%\newblock {\em Glasg. Math. J.}, 41(2):157--165, 1999.

\bibitem{Olshanskii}
A.Yu. Ol'shanskij.
\newblock {The SQ-universality of hyperbolic groups.}
\newblock {\em Sb. Math.}, 186(8):1199--1211, 1995.

\bibitem{PrideGlasgow}
S. J. Pride.
\newblock {Star-complexes, and the dependence problems for hyperbolic
  complexes.}
\newblock {\em Glasg. Math. J.}, 30(2):155--170, 1988.

\bibitem{Sela}
Z. Sela.
\newblock Endomorphisms of hyperbolic groups. {I}: The {H}opf property.
\newblock {\em Topology}, 38(2):301--321, 1999.

\bibitem{Simon}
J. Simon.
\newblock {Wirtinger approximations and the knot groups of $F\sp n$ in
  $S\sp{n+2}$.}
\newblock {\em Pac. J. Math.}, 90:177--190, 1980.

\bibitem{SV}
A. Szczepa\'nski and A. Vesnin.
\newblock {HNN} extension of cyclically presented groups.
\newblock {\em J. Knot Theory Ram.}, 10(8):1269--1279, 2001.

\bibitem{Maldeghem}
H. van Maldeghem.
\newblock {\em Generalized polygons}, volume~93 of {\em Monographs in
  Mathematics}.
\newblock Birkhauser Verlag, 1998.

\bibitem{Whitehead41}
J.H.C. Whitehead.
\newblock {On adding relations to homotopy groups.}
\newblock {\em Ann. Math.}, 42(2):409--428, 1941.

\bibitem{Wsurvey}
G. Williams.
\newblock {Groups of Fibonacci type revisited.}
\newblock {\em To appear in Proceedings of the fifth conference on Geometric and Asymptotic Group Theory with Applications (GAGTA-5), held at UPC-Manresa, July 2011.}
\end{thebibliography}
\end{document}